\documentclass[10pt]{amsart}  
\usepackage{amssymb} 
\usepackage{vector} 
\usepackage{tikz}
\usepackage{multirow} 
\usepackage{rotating}
\usetikzlibrary{arrows}
\newtheorem{proposition}{Proposition}[section]
\newtheorem{lemma}[proposition]{Lemma}
\newtheorem{theorem}[proposition]{Theorem}
\newtheorem{corollary}[proposition]{Corollary}

\newtheorem{example}[proposition]{Example}
\newcommand{\adj}[1]{\mathop{\rm adj}(#1)}  
\newcommand{\reals}{\mathbb{R}}               

\newcommand{\N}{\mathcal{N}}                
\newcommand{\A}{\mathcal{A}}
\newcommand{\Q}{\mathcal{Q}}
\newcommand{\D}{\mathcal{D}}
\newcommand{\T}{\mathcal{T}}
\newcommand{\Z}{\mathcal{Z}}
\newcommand{\cB}{\mathcal{B}}       
\newcommand{\cW}{\mathcal{W}}

\newcommand{\C}{\mathcal{C}}
\newcommand{\LL}{\mathcal{L}}

\begin{document}

\tikzstyle{place}=[circle,draw=black!100,fill=black!100,thick,
inner sep=0pt,minimum size=1mm]
\tikzstyle{left}=[>=latex,<-,semithick]
\tikzstyle{right}=[>=latex,->,semithick]
\tikzstyle{nleft}=[>=latex,-,semithick]
\tikzstyle{nright}=[>=latex,-,semithick]

\title[Potentially Nilpotent Patterns and the Nilpotent-Jacobian Method]{Potentially Nilpotent Patterns\\ and the Nilpotent-Jacobian Method}
\author{Hannah Bergsma}  
\address{Department of Mathematics, Redeemer University College, Ancaster, ON
L9K 1J4 Canada}
\email{hbergsma@redeemer.ca}

\author{Kevin N. Vander Meulen}  
\address{Department of Mathematics, Redeemer University College, Ancaster, ON
L9K 1J4 Canada}
\email{kvanderm@redeemer.ca}
\urladdr{http://cs.redeemer.ca/math/kvhome.htm}

\author{Adam Van Tuyl}  
\address{Department of Mathematical Sciences \\  
Lakehead University \\  
Thunder Bay, ON P7B 5E1, Canada}  
\email{avantuyl@lakeheadu.ca}  
\urladdr{http://flash.lakeheadu.ca/$\sim$avantuyl/}  
  
\keywords{sign pattern matrix, nilpotent index, potentially nilpotent, spectrally aribitrary}  
\subjclass[2000]{15A18, 15A29, 05C50,  15B35 }  
\thanks{Version: \today}  

\begin{abstract}
A nonzero pattern is a matrix with entries in $\{0,*\}$. A pattern is potentially
nilpotent if there is some nilpotent real matrix with nonzero entries in precisely the 
entries indicated by the pattern. We develop ways to construct some potentially nilpotent
patterns, including some balanced tree patterns. 
We explore the index of some of the nilpotent matrices constructed,
and observe that some of the balanced trees are spectrally arbitrary 
using the Nilpotent-Jacobian method. 
Inspired by an argument in
[R. Pereira, Nilpotent matrices and spectrally arbitrary sign patterns. \emph{Electron. J. Linear Algebra}  {\bf 16}  (2007), 232--236], we also uncover a feature of the 
Nilpotent-Jacobian method.  In particular, we show that if $N$ is the
nilpotent matrix employed in this method to show that a pattern
is a spectrally arbitary pattern, then $N$ must have full index.

\end{abstract}

\maketitle

\section{Introduction and Terminology}
A (zero-nonzero) {\emph{pattern}} $\A$ is a square matrix whose entries come from
the set $\{*,0\}$ where $*$ denotes a nonzero entry.  
We then set 
\[Q(\A) = \{ A \in M_n(\reals) ~|~ (A)_{i,j} \neq 0 \Leftrightarrow 
(\A)_{i,j} = * ~~\mbox{for all $i,j$}\}.\]
An element $A \in Q(\A)$
is called a {\it matrix realization} of $\A$. 
A \emph{sign pattern} $\A$ is a square matrix whose entries come from
the set $\{ +,-,0\}$. While the results in this paper are written
in terms of (zero-nonzero) patterns, each of the results also
apply to sign patterns. 

A pattern $\A$ is said to be {\it potentially
nilpotent} if there exists a matrix $A \in Q(\A)$ such that 
$A^k = 0$ for some positive integer $k$, that is, if there is a nilpotent
matrix in $Q(\A)$.  In Sections~\ref{construction} and~\ref{construction2} we describe some
new constructions of potentially nilpotent patterns. The need for constructions of
potentially nilpotent patterns was described in~\cite{EJ}; the recent 
paper~\cite{Catral} summarizes much of the current knowledge about potentially
nilpotent patterns.

If $N$ is a nilpotent matrix, then the {\emph{index}} of $N$
is the smallest positive integer $k$ such that $N^k = 0$. We say
a pattern $\N$ \emph{allows index $k$} if there is some nilpotent matrix $N\in Q(\N)$ 
with index $k$. We say an $n \times n$ pattern $\N$ \emph{allows full index} if
$\N$ allows index $n$. Recent work~\cite{EK} has focused on indices allowed by
(sign) patterns.

A pattern $\A$ is \emph{spectrally arbitrary} if any self-conjugate multiset of complex numbers is the spectrum
of some $A\in Q(\A)$. Colloquially, $\A$ is spectrally arbitrary if $\A$ allows any set of eigenvalues.
If a pattern is spectrally arbitrary then it is also potentially nilpotent but the converse
is not true. 

A pattern $\A$ is \emph{symmetric} if $\A^T=\A$. 
Note that it is
possible to have a nonsymmetric matrix $A\in Q(\A)$ when 
$\A$ is symmetric; $A\in Q(\A)$ is said to be combinatorially symmetric if $\A$ is symmetric. If a pattern $\A$ is symmetric, 
the graph of $\A$, denoted $G(\A)$, is the simple graph of the adjaceny matrix
obtained from $\A$ by replacing each $*$ by $1$. A pattern $\A$ is said to
have a loop at vertex $i$ if $\A_{i,i}\neq 0$.
Some of the constructions in this paper result in potentially nilpotent
symmetric patterns whose graphs are trees. 
Tridiagonal patterns of order $n$
$$\T_n=
\begin{bmatrix}
*     &*     &0     &\cdots&0\\
*     &0     &\ddots&\ddots&\vdots\\
0     &\ddots&\ddots&\ddots&0   \\
\vdots&\ddots&\ddots&0    &* \\      
0     &\cdots&0     &*     &*  
\end{bmatrix},
$$
correspond to a path graph on $n$ vertices with loops on each leaf. 
These patterns are known~\cite{Drew,Elsner} to be spectrally arbitray for
$2\leq n \leq 16$, and are conjectured~\cite{Drew} to be spectrally aribitrary
(and hence potentially nilpotent) for all $n\geq 2$. 
Patterns of order $n$ whose graph is a star with loops at each leaf,
$$\Z_n=
\begin{bmatrix}
0&*&\cdots&\cdots&*\\
*&*&0&\cdots&0\\
\vdots&0&\ddots&\ddots&\vdots\\
\vdots&\vdots&\ddots&\ddots&0\\
*&0&\cdots&0&*\\
\end{bmatrix},
$$
are known~\cite{MTvdD} to be spectrally arbitrary and hence potentially nilpotent 
for all $n\geq 3$. 
In Section~\ref{PNtree}, we describe some potentially nilpotent tree
patterns which generalize the star patterns.

Determining if a pattern is potentially nilpotent is often a
key step in determining if a pattern is spectrally arbitrary. 
The paper~\cite{Drew} initiated the study of spectrally arbitary 
patterns and described a method for determining that a pattern is spectrally arbitrary called the Nilpotent-Jacobian method (see Section~\ref{full} below).
Finding an appropriate nilpotent matrix is the first step in the method and has been described
as the Achilles' heel of the Nilpotent-Jacobian method in~\cite{CaversKimShader}.
Pereira~\cite{P}, using the Nilpotent-Jacoiban method, 
demonstrated that certain patterns (with lots of nonzero entries) that allow full index were spectrally arbitrary. In Section~\ref{full}
we demonstrate the necessity of the full index hypothesis for successfully
employing the Nilpotent-Jacobian method on a spectrally aribitrary pattern (regardless of
the quantity of nonzero entries):

\textbf{Theorem~\ref{big}}: {\emph{Let $\A$ be a $n\times  n$ pattern. Suppose that one can show that $\A$ is spectrally arbitrary using
the Nilpotent-Jacobian method. If $N$ is the nilpotent matrix in $Q(\A)$ used in the Nilpotent-Jacobian
method, then $N$ has index $n$.}}

\section{A Construction for Potentially Nilpotent Patterns}\label{construction}

In this section we introduce a way to construct new potentially nilpotent patterns based upon
known potentially nilpotent patterns.
Let $\A_{m}(\N)$ be the $n \times n$ pattern 
\begin{equation}\label{pattern}
\A_{m}(\N)=\begin{bmatrix} 
0 		& \ell^{T} 	&		& \cdots 	& \ell^{T}	\\
\ell    & \N	& 0		& \cdots& 0		\\
\ell   & 0 	& \ddots& \ddots&\vdots	\\
\ell  	& \vdots& \ddots&		& 0		\\
\ell & 0 	& \cdots&	0	& \N	
\end{bmatrix}
\end{equation}
where $\ell^T=[* ~ 0 \cdots 0]$, $\N$ is a $s \times s$ pattern and $n=ms+1$.

\begin{figure}[h]
\begin{tikzpicture}
\node (1) at (0,0)[place]{};
\node (2) at (0,0.5)[place]{};
\node (3) at (0,-0.5)[place]{};
\node (4) at (-0.353,0.853)[place]{};
\node (5) at (0.353,0.853)[place]{};
\node (6) at (-0.353,-0.853)[place]{};
\node (7) at (0.353,-0.853)[place]{};
\draw [] (1) to (2);
\draw [] (1) to (3);
\draw [] (2) to (4);
\draw [] (2) to (5);
\draw [] (3) to (6);
\draw [] (3) to (7);
\end{tikzpicture}
\hspace{15mm}
\begin{tikzpicture}
\node (1) at (0,0)[place]{};
\node (2) at (0,0.5)[place]{};
\node (3) at (0,-0.5)[place]{};
\node (4) at (-0.353,0.853)[place]{};
\node (5) at (0,1)[place]{};
\node (6) at (0.353,0.853)[place]{};
\node (7) at (-0.353,-0.853)[place]{};
\node (8) at (0.353,-0.853)[place]{};
\node (9) at (0,-1)[place]{};
\draw [] (1) to (2);
\draw [] (1) to (3);
\draw [] (2) to (4);
\draw [] (2) to (5);
\draw [] (2) to (6);
\draw [] (3) to (7);
\draw [] (3) to (8);
\draw [] (3) to (9);
\end{tikzpicture}
\hspace{15mm}
\begin{tikzpicture}
\node (1) at (0,0)[place]{};
\node (2) at (0,0.5)[place]{};
\node (3) at (0,-0.5)[place]{};
\node (4) at (-0.2939,0.9045)[place]{};
\node (5) at (-0.4755,0.6546)[place]{};
\node (6) at (0.2939,0.9045)[place]{};
\node (7) at (0.4755,0.6546)[place]{};
\node (8) at (-0.2939,-0.9045)[place]{};
\node (9) at (-0.4755,-0.6546)[place]{};
\node (10) at (0.2939,-0.9045)[place]{};
\node (11) at (0.4755,-0.6546)[place]{};
\draw [] (1) to (2);
\draw [] (1) to (3);
\draw [] (2) to (4);
\draw [] (2) to (5);
\draw [] (2) to (6);
\draw [] (2) to (7);
\draw [] (3) to (8);
\draw [] (3) to (9);
\draw [] (3) to (10);
\draw [] (3) to (11);
\end{tikzpicture}
\caption{The graphs $G(\A_2(\Z_3))$, $G(\A_2(\Z_4))$ and $G(\A_2(\Z_5))$.}\label{graphs}
\end{figure}

Let $A_m(N_1,N_2,\ldots,N_m)\in \Q(\A_m(\N))$ be the $n \times n$ matrix 
\begin{equation}\label{matrix}
A_{m}(N_1,N_2,\ldots,N_m)=\begin{bmatrix} 
0 		& \bvec{e}^{T} 	&		& \cdots 	& \bvec{e}^{T}	\\
\bvec{a}& N_{1}	& 0		& \cdots& 0		\\
\bvec{e}& 0 	& \ddots& \ddots&\vdots	\\
\vdots 	& \vdots& \ddots&		& 0		\\
\bvec{e}& 0 	& \cdots&	0	& N_{m}	
\end{bmatrix}
\end{equation}
where $N_{1}, N_{2}, \ldots ,N_{m}$ are $s\times s$ nilpotent matrices, 
$m \geq 2$,  $\bvec{e}^T=\begin{bmatrix}1&0&\cdots&0\end{bmatrix}$, and $\bvec{a}=(1-m)\bvec{e}$.
In the results developed in this
paper, we explore nilpotent realizations $A_m(N_1,\ldots,N_m)\in Q(\A_m(\N))$ with
$N_1=N_2=\cdots=N_m$. We will adopt the notation $A_m(N)$ to denote the matrix $A_m(N,N,\ldots,N)$.
The patterns $\A_m(\T_2)$ with $m\geq 3$ are called 2-generalized star patterns in~\cite{GS}.


While the next theorem follows from the more detailed result~Theorem~\ref{index}, we
include a direct argument here. The patterns $\A_2(\Z_3), \A_2(\Z_4)$ and $\A_2(\Z_5)$
with graphs in Fig.~\ref{graphs} satisfy this theorem; in fact we will see
in Section~\ref{full} that they are spectrally arbitrary patterns.

\begin{theorem}\label{PN}
If $m\geq 2$ and $\N$ is potentially nilpotent, then $\A_m(\N)$ is potentially nilpotent.
\end{theorem} 

\begin{proof} Let $N\in Q(\N)$ be a nilpotent matrix. Consider the matrix $A=A_{m}(N)\in Q(\A_m(\N))$. 
Let $N'$ be the matrix obtained from $N$ by replacing the first row with $\bvec{e}^T=\begin{bmatrix} 1&0&\cdots&0\end{bmatrix}.$
Then we can determine $p_A=p_A(x)$, the characteristic polynomial of $A$, by cofactor expansion along the first column. Note
that for $0\leq k\leq (m-1)$, the $(ks+2,1)$ cofactor of $xI-A$ is 
$$(-1)^{(ks+2)+1}\det(xI-A)_{ks+2,1}$$ where the $(ks+2,1)$-minor is $$\det(xI-A)_{ks+2,1}=(-1)^{ks}(p_{N^{'}})(p_{N})^{m-1}.$$
Thus
\begin{align*}
p_A(x) & = \det(xI-A)\\
		& = x(p_{N})^{m} + (-1)^{2+1}(-a)(p_{N^{'}})(p_{N})^{m-1} + (m-1)(p_{N^{'}})(p_{N})^{m-1}  \\
		& = x(p_{N})^{m} \mbox{\rm{\quad since $a=(1-m)$ \quad}}\\
		& = x^n  \mbox{\rm{\quad since $N$ is nilpotent \quad}} 
\end{align*}
Thus $A$ is nilpotent and so $\A_m$ is potentially nilpotent.
\end{proof}

Note that if $\N$ is a potentially nilpotent pattern, then so is $P\N P^T$ for
any permutation matrix $P$, but $\A_m(P\N P^T)$ is not permutationally equivalent
to $\A_m(P\N P^T)$ in general. Thus Theorem~\ref{PN} provides a wide class
of potentially nilpotent patterns.

We will next develop an observation about the index of the matrix $A_m(N)$ when $N$ is nilpotent.

\begin{lemma}\label{powerlemma} Suppose $m\geq 2$, 
and $A=A_{m}(N)$. 
Then for all $t\geq 1$, 
\begin{equation}\label{expo}
A^{t}= 
\left[\begin{array}{cc}
0 & \bvec{f}_t  \\ 
\bvec{g}_t& B_t		 \\
\end{array}\right]
\end{equation}
where $\bvec{f}_t=\begin{bmatrix}\bvec{e}^{T}N^{t-1}&\cdots&\bvec{e}^{T}N^{t-1}\end{bmatrix}$,
$$ \bvec{g}_t =\begin{bmatrix}
aN^{t-1}\bvec{e} \\ N^{t-1}\bvec{e} \\ \vdots \\ N^{t-1}\bvec{e}
\end{bmatrix}, \qquad 
B_t=(-m) \begin{bmatrix} W_{t} & \cdots & W_{t} \\
& &\\
& 0 &
\end{bmatrix}
+ \begin{bmatrix}
W_{t} & \cdots & W_{t}\\
\vdots & \ddots & \vdots\\
W_{t} & \cdots & W_{t}
\end{bmatrix}
+  \begin{bmatrix} N^{t} & & 0 \\
& \ddots & \\
0&& N^{t}
\end{bmatrix}, 
$$
$a=(1-m)$, $W_{t}=\displaystyle\sum\limits_{i=0}^{t-2} N^{i}\bvec{e}\bvec{e}^{T}N^{(t-2)-i}$ for $t\geq 2$ and $W_1=0$.
\end{lemma}

\begin{proof} We proceed by induction on $t$. Observe that $A=A^1$ satisfies (\ref{expo}) since
$W_1=0$. 
To minimize notation, we let  
$L= \bvec{e}\bvec{e}^{T} =\begin{bmatrix}1&\\
&0 
\end{bmatrix}. $

We observe that $A^{t+1}=AA^t$ satisfies (\ref{expo}) because
$$\bvec{f}_1\bvec{g}_t = a \bvec{e}^TN^{t-1}\bvec{e}+(m-1)\bvec{e}^TN^{t-1}\bvec{e}=0,$$
$\bvec{f}_1B_t=\bvec{f}_{t+1}$, $B_1\bvec{g}_t=\bvec{g}_{t+1}$ and
$$
\bvec{g}_1\bvec{f}_t+B_1B_t=
\begin{bmatrix} 
aLN^{t-1}&\cdots&aLN^{t-1}\\
LN^{t-1} & \cdots & LN^{t-1}\\
\vdots & \ddots & \vdots\\
LN^{t-1} & \cdots & LN^{t-1}
\end{bmatrix}
+  \begin{bmatrix} N & & 0 \\
& \ddots & \\
0&& N
\end{bmatrix}B_t=B_{t+1}
$$
since $LN^{t-1}+NW_{t}=W_{t+1}$.
\end{proof}

\begin{theorem}\label{index}
Suppose $m\geq 2$, $A=A_{m}(N),$ 
and $N$ is nilpotent of index $k$. Then $A$ is nilpotent of index $r$ for some $r\leq 2k+1$.
Further, if both the first row and column of $N^{k-1}$ contain at least one nonzero entry,
then $A$ has index  $r=2k+1$. 
\end{theorem}

\begin{proof} Let $L= \bvec{e}\bvec{e}^{T}$. Since $N^k=0$, it follows from Lemma~\ref{powerlemma} that 
$A^{2k}=\begin{bmatrix}
0&\mathbf{0}^T\\
\mathbf{0}& B_{2k}
\end{bmatrix}$
with 
$$B_{2k}=(-m) \begin{bmatrix} W_{2k} & \cdots & W_{2k} \\
& &\\
& 0 &
\end{bmatrix}
+\begin{bmatrix}
W_{2k} & \cdots & W_{2k}\\
\vdots & \ddots & \vdots\\
W_{2k} & \cdots & W_{2k}
\end{bmatrix}.$$
But $W_{2k}=N^{k-1}LN^{k-1} \neq 0$ if the first row and column of $N^{k-1}$ both contain a nonzero entry.
Hence $A^{2k}\neq 0$, but
$A^{2k+1}=0$ since $W_{2k+1}=0$. 
\end{proof}

\section{Another Construction for Potentially Nilpotent Patterns}\label{construction2}

We next describe a slight variation of the general construction given
in Section~\ref{construction}, to give another construction which builds potentially
nilpotent patterns from known potentially nilpotent patterns.

Let $\N$ by an $s\times s$ pattern, and $\C_m(\N)$ be the $ms \times ms$ pattern
$$\C_m(\N)=\begin{bmatrix} 
\N 		& \LL		&\cdots&\cdots & \LL\\
\LL  	& \N 	&0 	   &\cdots &0		\\
\LL       &0      &\ddots&\ddots &\vdots        \\
\vdots  & \vdots&\ddots&\ddots &0     \\
\LL       & 0     &\cdots&0       & \N
\end{bmatrix}
$$
where $\LL=\ell\ell^T$ and $\ell=\begin{bmatrix}*&0&\cdots&0\end{bmatrix}$.
The pattern $\C_2(\Z_s)$ is an example of a \emph{double star pattern}. 
A recent paper~\cite{LiLi} discusses the potential nilpotence of some
double star (sign) patterns.

Let $N$ by an $s\times s$ matrix, $m\geq 3$ and $C_m(N)\in Q(\C_m(\N))$ be the $ms \times ms$ matrix
$$C_m(N)=\begin{bmatrix} 
N 		& L		&\cdots&\cdots & L\\
aL  	& N 	&0 	   &\cdots &0		\\
L       &0      &\ddots&\ddots &\vdots        \\
\vdots  & \vdots&\ddots&\ddots &0     \\
L       & 0     &\cdots&0       & N
\end{bmatrix}
$$
where $a=2-m$, and $L=\bvec{e}\bvec{e}^T.$

\begin{lemma}\label{powerC} 
Suppose $m\geq 3$, $R_2=R_1=0$, $W_1=W_0=0$, and
$W_{t}=\displaystyle\sum\limits_{i=0}^{t-2} N^{i}LN^{(t-2)-i}$ for $t\geq 2$. 
If $C=C_m(N)$, then for all $t\geq 1$,
\begin{equation}\label{expo2}
C^t=
 \begin{bmatrix} 
 N^{t} & & 0 \\
 & \ddots & \\
 0&& N^{t}
 \end{bmatrix} 
+ \begin{bmatrix}
 0 & W_{t+1}&\cdots & W_{t+1}\\
 aW_{t+1} & aK_t &\cdots & aK_t\\
 W_{t+1}  & K_t    &\cdots  & K_t  \\
 \vdots   & \vdots    &\ddots  &\vdots    \\   
 W_{t+1} & K_t     & \cdots &K_t \\
 \end{bmatrix}
\end{equation}
where $K_t=LW_t+R_t$, and for $t\geq 3$,
\begin{equation}\label{R}
R_{t}=\displaystyle\sum\limits_{i=0}^{t-3}\displaystyle\sum\limits_{j=0}^{t-3-i} N^{j+1}LN^{(t-3)-j-i}LN^{i}.
\end{equation}
\end{lemma}

\begin{proof}Let $C=C_m(N)$.
We sketch a proof by induction on $t$. Note that $C=C^1$ has the form~(\ref{expo2}) since
$K_1=LW_1+R_1=0$ and $W_2=L$.
To check that $C^{t+1}$ is of form~(\ref{expo2}) one can use the induction 
hypothesis with 
$C^{t+1}=CC^{t}$, noting that
$W_{t+2}=NW_{t+1}+LN^t$ and $R_{t+1}=N(LW_t+R_t)$.
\end{proof}

\begin{theorem} If $m\geq 3$ and $\N$ is potentially nilpotent, then
$\C_m(\N)$ is potentially nilpotent. In fact, if $\N$
allows index $k$, then
$\C_m(\N)$ allows index $r$ for some $r\leq 3k$.
\end{theorem}

\begin{proof} Let $C=C_m(N)\in Q(\C_m(\N))$ for some 
matrix $N$ which is nilpotent of index $k$.

Note that each summand in $R_t$ (from line (\ref{R})) is of the form $N^aLN^bLN^c$ such that $a+b+c=t-2$.  
If $t-2 \geq 3k-2$, then $a \geq k$ or $b \geq k$ or $c \geq k$.
Hence $R_t=0$ when $t \geq 3k$. 

Since $W_t=0$ for $t>2k$, Lemma~\ref{powerC} implies that $C^{3k}=0.$ 
 \end{proof}

\section{Potentially Nilpotent Balanced Tree Patterns}\label{PNtree}

In this section we explore some examples of tree patterns which can be shown to be 
potentially nilpotent using the construction in Section~\ref{construction}.
We call a tree $T$ {\it balanced} if there exists a designated root $r$ with deg$(r) \geq 2$ and  deg$(x)=$deg$(y)$ for all $x,y \in V(T)$ with dist$(x,r)=$dist$(y,r)$. 
In this paper, we call a pattern $\A$ a \emph{balanced tree pattern} if $G(\A)$ is a balanced tree
 and a vertex of $G(\A)$ has a loop if and only if the vertex is a leaf of $G(\A)$. 
Note that $\T_n$ is a balanced tree pattern for $n$ odd (by designating
the middle vertex of the path as the root).
Likewise, for $n\geq 3$, the star patterns $\Z_n$ with a loop at each leaf and no loop at the root 
are examples of balanced tree patterns. The star pattern $\Z_n$ has been shown to be potentially nilpotent for $n\geq 3$ (see~\cite[Theorem 3.1]{MTvdD}).
We demonstrate that other balanced tree patterns are also potentially nilpotent. 

\begin{figure}[h]
\begin{tikzpicture}
\node (1) at (0,0)[place]{};
\node (2) at (0,0.45)[place]{};
\node (3) at (0,0.65)[place]{};
\node (4) at (-0.177,0.577)[place]{};
\node (5) at (0.177,0.577)[place]{};
\node (6) at (-0.1765,0.8265)[place]{};
\node (7) at (0.1765,0.8265)[place]{};
\node (8) at (-0.3538,0.7537)[place]{};
\node (9) at (-0.457,0.577)[place]{};
\node (10) at (0.3538,0.7537)[place]{};
\node (11) at (0.457,0.577)[place]{};
\node (12) at (0.45,0)[place]{};
\node (13) at (0.65,0)[place]{};
\node (14) at (0.5768,0.1768)[place]{};
\node (15) at (0.5768,-0.1768)[place]{};
\node (16) at (0.8268,0.1768)[place]{};
\node (17) at (0.8268,-0.1768)[place]{};
\node (18) at (0.577,0.457)[place]{};
\node (19) at (0.7537,0.3538)[place]{};
\node (20) at (0.577,-0.457)[place]{};
\node (21) at (0.7537,-0.3538)[place]{};
\node (22) at (0,-0.45)[place]{};
\node (23) at (0,-0.65)[place]{};
\node (24) at (-0.177,-0.577)[place]{};
\node (25) at (0.177,-0.577)[place]{};
\node (26) at (-0.1765,-0.8265)[place]{};
\node (27) at (0.1765,-0.8265)[place]{};
\node (28) at (-0.3538,-0.7537)[place]{};
\node (29) at (-0.457,-0.577)[place]{};
\node (30) at (0.3538,-0.7537)[place]{};
\node (31) at (0.457,-0.577)[place]{};
\node (32) at (-0.45,0)[place]{};
\node (33) at (-0.65,0)[place]{};
\node (34) at (-0.5768,0.1768)[place]{};
\node (35) at (-0.5768,-0.1768)[place]{};
\node (36) at (-0.8268,0.1768)[place]{};
\node (37) at (-0.8268,-0.1768)[place]{};
\node (38) at (-0.577,0.457)[place]{};
\node (39) at (-0.7537,0.3538)[place]{};
\node (40) at (-0.577,-0.457)[place]{};
\node (41) at (-0.7537,-0.3538)[place]{};
\draw [] (1) to (2);
\draw [] (2) to (3);
\draw [] (2) to (4);
\draw [] (2) to (5);
\draw [] (3) to (6);
\draw [] (3) to (7);
\draw [] (4) to (8);
\draw [] (4) to (9);
\draw [] (5) to (10);
\draw [] (5) to (11);
\draw [] (1) to (12);
\draw [] (12) to (13);
\draw [] (12) to (14);
\draw [] (12) to (15);
\draw [] (13) to (16);
\draw [] (13) to (17);
\draw [] (14) to (18);
\draw [] (14) to (19);
\draw [] (15) to (20);
\draw [] (15) to (21);
\draw [] (1) to (22);
\draw [] (22) to (23);
\draw [] (22) to (24);
\draw [] (22) to (25);
\draw [] (23) to (26);
\draw [] (23) to (27);
\draw [] (24) to (28);
\draw [] (24) to (29);
\draw [] (25) to (30);
\draw [] (25) to (31);
\draw [] (1) to (32);
\draw [] (32) to (33);
\draw [] (32) to (34);
\draw [] (32) to (35);
\draw [] (33) to (36);
\draw [] (33) to (37);
\draw [] (34) to (38);
\draw [] (34) to (39);
\draw [] (35) to (40);
\draw [] (35) to (41);
\end{tikzpicture}
\hspace{15mm}
\begin{tikzpicture}
\node (1) at (0,0)[place]{};
\node (2) at (0,0.25)[place]{};
\node (3) at (-0.177,0.427)[place]{};
\node (4) at (0.177,0.427)[place]{};
\node (5) at (-0.3538,0.6037)[place]{};
\node (6) at (0.3538,0.6037)[place]{};
\node (7) at (0.25,0)[place]{};
\node (8) at (0.427,0.177)[place]{};
\node (9) at (0.427,-0.177)[place]{};
\node (10) at (0.6037,0.3538)[place]{};
\node (11) at (0.6037,-0.3538)[place]{};
\node (12) at (0,-0.25)[place]{};
\node (13) at (-0.177,-0.427)[place]{};
\node (14) at (0.177,-0.427)[place]{};
\node (15) at (-0.3538,-0.6037)[place]{};
\node (16) at (0.3538,-0.6037)[place]{};
\node (17) at (-0.25,0)[place]{};
\node (18) at (-0.427,0.177)[place]{};
\node (19) at (-0.427,-0.177)[place]{};
\node (20) at (-0.6037,0.3538)[place]{};
\node (21) at (-0.6037,-0.3538)[place]{};
\draw [] (1) to (2);
\draw [] (2) to (3);
\draw [] (2) to (4);
\draw [] (3) to (5);
\draw [] (4) to (6);
\draw [] (1) to (7);
\draw [] (7) to (8);
\draw [] (7) to (9);
\draw [] (8) to (10);
\draw [] (9) to (11);
\draw [] (1) to (12);
\draw [] (12) to (13);
\draw [] (12) to (14);
\draw [] (14) to (16);
\draw [] (13) to (15);
\draw [] (1) to (17);
\draw [] (17) to (18);
\draw [] (17) to (19);
\draw [] (18) to (20);
\draw [] (19) to (21);
\end{tikzpicture}
\caption{Graphs of balanced trees: $\A_4(\A_3(\Z_3))$ and $\A_4(\T^P_5)$.}\label{bush}
\end{figure}

Note that the pattern of $\A_m(\N)$ can be used to construct some balanced tree patterns recursively. 
We say that  
$\A$ is a \emph{recursive star pattern} 
 if $\A=\A_m(\N)$ and $\N$ is a recursive star pattern (with root at vertex 1),
or $\A=\Z_{s}$ is a star pattern with $s\geq 3$. The graph of $\A_4(\A_3(\Z_3))$ in Figure~\ref{bush}
is a recursive star pattern, but $\A_4(\T^P_5)$ is not a recursive star pattern
where 
$$\T^P_5=P\T_5P^T=
\begin{bmatrix}
0&*&0&*&0\\
*&0&*&0&0\\
0&*&*&0&0\\
*&0&0&0&*\\
0&0&0&*&*
\end{bmatrix}, \qquad {\mbox{\rm{with \quad}}}
P=\begin{bmatrix}
0&0&1&0&0\\
0&1&0&0&0\\
1&0&0&0&0\\
0&0&0&1&0\\
0&0&0&0&1
\end{bmatrix}.
$$

The next result is a corollary of Theorem~\ref{PN}.

\begin{theorem}\label{recursive}
If $\A$ is a recursive star pattern, 
then $\A$ is potentially nilpotent.
\end{theorem}

Theorem~\ref{recursive} gives a whole class of potentially nilpotent balanced trees but there are also other
potentially nilpotent balanced trees. For example the balanced tree pattern $\A_4(\T^P_5)$ (represented by the second
graph in Figure~\ref{bush}) is a potentially nilpotent balanced tree by Theorem~\ref{PN} since $\T^P_5$ is potentially nilpotent 
($\T^P_5$ is permutationally equivalent to $\T_5$).

\begin{theorem}\label{allowindex}
If the star pattern $\Z=\Z_s$ allows index $k$, 
then $\A_{m}(\Z)$ allows index $2k+1$, and
each row and column of $\Z^{k-1}$ has a nonzero entry.
\end{theorem}

\begin{proof} Let $\N$ be a star pattern of order $s$ and let $N\in Q(\N)$ have
index $k$. 
By Theorem~\ref{index} it is enough to show that each row and column of $N^{k-1}$ contains a nonzero entry. 
Suppose that row $1$ of $N^{k-1}$ is $\bvec{0}^T$. Since $N$ has index $k$, $N^{k-1}_{ij}\neq 0$ 
for some row $i\neq 1$ and some column $j$. Then $N^k_{ij}\neq 0$ since $N^k_{ij}=(NN^{k-1})_{ij}=N_{ii}N^{k-1}_{ij}\neq 0$.
This contradicts the fact that $N^k=0$. Thus row $1$ of $N^{k-1}$ is not $\bvec{0}^T$ 
 and contains a nonzero entry. 

Suppose that for some $i$, $2\leq i \leq s$, row $i$ of $N^{k-1}$ is $\bvec{0}^T$. Then since $N^k=0$ it follows that
for each column $j$, 
$0=N^k_{ij}=(NN^{k-1})_{ij}=N_{i1}N^{k-1}_{1j}+N_{ii}N^{k-1}_{ij}=N^{k-1}_{1j}$.
That is, row $1$ of $N^{k-1}$ is $\bvec{0}^T$. But this contradicts what was observed
above. Thus every row 
of $N^{k-1}$ contains
a nonzero entry.

Since $\N$ is symmetric and $N^k=N^{k-1}N$, a similar argument show that every column of $N^{k-1}$ contains a nonzero entry.
\end{proof}

The following corollary follows from Theorem~\ref{index} and Theorem~\ref{allowindex}.

\begin{corollary}\label{starfull}
If the star pattern $\Z_s$ allows full index $s$, then $\A_{2}(\Z_s)$ allows full index.
\end{corollary}

We will explore the significance of allowing full index in Section~\ref{full}.

\section{Spectrally arbitrary patterns, full index, and the Nilpotent-Jacobian method.}\label{full}

In this section we observe that it is necessary that a nilpotent matrix have
full index if it is to be successfully employed in the Nilpotent-Jacobian method
to demonstrate that a pattern is spectrally arbitrary. We will first describe
the Nilpotent-Jacobian method and give some examples. 

Let $\A$ be an $n \times n$ sign pattern, and suppose that
there exists a nilpotent matrix $N \in Q(\A)$ with at least
$n$ non-zero entries, say $a_{i_1,j_1},\ldots,a_{i_n,j_n}$.  Let
$X=X(x_1,\ldots,x_n)$ denote the matrix obtained from $N$ with the $(i_k,j_k)$-th position
replaced with the variable $x_k$ for $k=1,\ldots,n$, and let
\[
{p}_{X}(x) = x^n + f_1x^{n-1} + f_2x^{n-2}+
\cdots + f_{n-1}x + f_n\]
be the characteristic polynomial of $X(x_1,\ldots,x_n)$, where each
$f_i = f_i(x_1,\ldots,x_n)$ is a polynomial in $x_1,\ldots,x_n$.
Let $J$ be the order $n \times n$ Jacobian matrix with
\[J = \begin{bmatrix}
\frac{\partial f_1}{\partial x_1} & \frac{\partial f_1}{\partial x_2} & \cdots & \frac{\partial f_1}{\partial x_n} \\
\frac{\partial f_2}{\partial x_1} & \ddots&  &\frac{\partial f_2}{\partial x_n} \\
\vdots & & & \vdots \\
\frac{\partial f_n}{\partial x_1} & \frac{\partial f_1}{\partial x_2} & \cdots & \frac{\partial f_n}{\partial x_n} \\
\end{bmatrix}.
\]
Setting  $\mathbf{x} = (x_1,\ldots,x_n)$ and $\mathbf{a}=(a_{i_1,j_1},\ldots,a_{i_n,j_n})$,
we let 
\begin{equation}\label{Jacob}
J'=J|_{\mathbf{x}=\mathbf{a}}
\end{equation}
 denote the Jacobian matrix evaluated at 
${\mathbf{x}=\mathbf{a}}$.
The \emph{\textbf{Nilpotent-Jacobian method}} is to seek a nilpotent realization of $\A$ for which 
$J'$ is nonsingular: the following theorem, first developed in~\cite{Drew}, shows
that $\A$ is spectrally arbitrary if such a realization is found.

Recall that $\cB$ is a \emph{superpattern} of a pattern $\A$ if $\A_{i,j}\neq 0$ implies $\cB_{i,j}\neq 0$.
Note that $\A$ is a superpattern of itself. 

\begin{theorem}[\cite{Drew}]\label{NJ}
If $J'$ is nonsingular, 
then every superpattern of $\mathcal{A}$ is spectrally arbitrary.
\end{theorem}

We apply the Nilpotent-Jacobian method to the balanced tree patterns $\A_2(\Z_3)$, $\A_2(\Z_4)$
and $\A_2(\Z_5)$ to demonstrate that these are spectrally aribitrary patterns.
The examples will refer
to the matrices $N\in Q(\A_2(\Z_s))$ of the form
\begin{equation}\label{greek}
N=\left[ \begin{array}{c|cccc|ccccc}
0		&1			&0			&\cdots	&0			&1		&0		&\cdots &\cdots	&0\\ \hline
x_1     &0			&1			&\cdots	&1			&		&		&		&&\\
0		&x_2		&x_{s+1}	&		&			&		&		&0	&&	\\
\vdots	&\vdots		&			&\ddots	&			&		&		&		&&\\
0		&x_{s}  	&			&		&x_{2s-1}   &		&		&		&&\\ \hline
y_1     &			&			&		&			&0		&1		&\cdots	&\cdots&1 \\
0		&			&			&		&			&x_{2s}	&x_{2s+1}&		&&\\
\vdots	&			&0		&		&			&y_2	&		 &y_{s}	&&\\
\vdots	&			&   		&		&			&\vdots	&		&&\ddots	&\\
0		&			&			&		&			&y_{s-1}&	&	&	&y_{2s-3}
\end{array}\right].
\end{equation}

\begin{example}\label{E1} 
Every superpattern of $\A_2(\Z_3)$ is spectrally arbitrary.
\end{example}
\begin{proof}
Take $N=A_2(Z)\in Q(\A_2(\Z_3))$ with 
$$Z=\left[\begin{array}{rrr}
0 & 1 & 1 \\
-{1}/{2} 
& 1 & 0\\
-{1}/{2} 
& 0 & -1
\end{array}\right]
$$ and $x_1=-y_1=-1$.
Since $Z$ is nilpotent, it follows from Theorem~\ref{index} and Corollary~\ref{allowindex} that $N$ is nilpotent.
Let $X$ be the matrix obtained from $N$ by setting 
$x_1,\ldots,x_7$ as variables. 
One can check that the Jacobian matrix $J'=J|_{(x_1,\ldots,x_7)=(-1,-{1}/{2},-{1}/{2},1, -1,-{1}/{2},1)}$
has a nonzero determinant.
Thus by Theorem~\ref{NJ},
every superpattern of $\A_2(\Z_3)$ is spectrally arbitary.
\end{proof}

\begin{example}\label{E2}
Every superpattern of $\A_2(\Z_4)$ is spectrally arbitrary.
\end{example}

\begin{proof}
  Take $N=A_2(Z)\in Q(\A_2(\Z_4))$ with 
  $$Z=\left[ \begin{array}{rrrr}
0 & 1 & 1&1 \\
{1}/{4} & 1 & 0&0\\
-{16}/{5} &0&2&0\\ 
-{81}/{20}&0&0&-3\\
\end{array}\right]
$$ and $x_1=-y_1=-1$.
Since $Z$ is nilpotent, it follows from Theorem~\ref{index} and Corollary~\ref{allowindex} that $N$ is nilpotent.
Let $X$ be obtained from $N$ by taking 
$x_1,\ldots,x_{9}$ as variables. 
One can check that the Jacobian matrix $J'=J|_{(x_1,\ldots,x_{9})=(-1,{1}/{4},-{16}/{5},-{81}/{20} ,1,2,-3,{1}/{4},1)}$
has nonzero determinant. Thus $\A_2(\Z_4)$ is a spectrally arbitrary pattern.
\end{proof}

\begin{example}\label{E3}
Every superpattern of $\A_2(\Z_5)$ is spectrally arbitrary.
\end{example}
 
 \begin{proof}  
  Take $N=A_2(Z)\in Q(\A_2(\Z_5))$ with 
  $$Z=\left[ \begin{array}{rrrrr}
0 & 1 & 1&1&1 \\
-{1}/{14} & 1 & 0&0&0\\
4 &0&2&0&0\\ 
-{27}/{2}&0&0&3&0\\
-{108}/{7}&0&0& 0 & -6
\end{array}\right]
$$ and $x_1=-y_1=-1$.
Since $Z$ is nilpotent, it follows from Theorem~\ref{index} and Corollary~\ref{allowindex} that $N$ is nilpotent.
Let $X$ be obtained from $N$ by taking 
$x_1,\ldots,x_{11}$ as variables. 
One can check that the Jacobian matrix 
$J'=J|_{(x_1,\ldots,x_{11})=(-1,-{1}/{14},4,-{27}/{2},-{108}/{7},1,2,3,-6,-{1}/{14},1)}$
has nonzero determinant. Thus $\A_2(\Z_5)$ is a spectrally arbitrary pattern.
\end{proof}

In Examples~\ref{E1}-\ref{E3}, the matrix $N$ that we constructed has full index. We then used
this matrix in the Nilpotent-Jacobian method. If one examines other cases where the 
Nilpotent-Jacobian method is used (see for example the pattern $\cW_n(k)$ in~\cite{Britz} and
$\D_{n,r}$ 
in~\cite{CV05}), one will also find the 
initial nilpotent matrix has full index. As we will see in Theorem~\ref{big}, this is not a coincidence, but
a necessary condition about nilpotent matrices used in the Nilpotent-Jacobian method.

We first formalize a couple of observations of
Pereira~\cite[proof of Theorem 2.2]{P} about the entries
of the $\adj{xI-N}$. For $1\leq i,j, \leq n$, we let
$$p_{i,j}(x)=[\adj{xI-N}]_{j,i}.$$

\begin{lemma}\label{P1}
The $k^{th}$ column of the Jacobian matrix $J'$ described in line~$(\ref{Jacob})$
consists of the coefficients of the polynomial $-p_{i_k,j_k}(x)$ for
all $1\leq k \leq n$.
\end{lemma}

\begin{proof} 

For $1\leq k \leq n$, let 
\[p_{i_k,j_k}(x) = c_{k,0} + c_{k,1}x + \cdots + c_{k,n-1}x^{n-1}.\]

Pick any $k$, $1\leq k\leq n$.
Thinking of $\det(xI-X)$ as a linear function of $x_k$, observe that $\frac{\partial{p}_{X}(x)}{\partial x_k}$ provides the coefficient
of $x_k$ in the expansion of $\det(xI-X)$. Note also that 
$$\frac{\partial{p}_{X}(x)}{\partial x_k}=\frac{\partial{f_1}}{\partial{x_k}}x^{n-1}
+\frac{\partial{f_2}}{\partial{x_k}}x^{n-2}+\cdots+\frac{\partial{f_n}}{\partial{x_k}}.$$
But this is the negative of the cofactor of position of $(i_k,j_k)$ in $(xI-X)$
since the entry in position $(i_k,j_k)$ is $-x_k$.
Thus 
$$\left.{\left.{\frac{\partial{p}_{X}(x)}{\partial x_k}}\right\vert_{\mathbf{x}=\mathbf{a}}
=-[\adj{xI-X}]_{j_k,i_k}}\right\vert_{\mathbf{x}=\mathbf{a}}
=-[\adj{xI-N}]_{j_k,i_k}
=-p_{i_k,j_k}(x).$$
Therefore
$$J'=J|_{\mathbf{x}=\mathbf{a}}=-\begin{bmatrix}c_{1,n-1} & c_{2,n-1} & \cdots & c_{n,n-1} \\
c_{1,n-2} & c_{2,n-2} & \cdots & c_{n,n-2} \\
\vdots & \vdots &\ddots & \vdots \\
c_{1,1} & c_{2,1}  & & c_{n,1} \\
c_{1,0} & c_{2,0} & \cdots & c_{n,0}
\end{bmatrix}.$$
\end{proof}

\begin{example} \emph{
Consider the nilpotent matrix $$N=\left[\begin{array}{rrr}
0&1&1\\
-1/2&1&0\\
-1/2&0&-1
\end{array}\right]
\qquad {\mbox{\rm{with \quad}}} 
X=\left[\begin{array}{rrr}
0&1&1\\
x_1&x_2&0\\
x_3&0&-1
\end{array}\right].
$$
Then $p_X(x)=x^3+(1-x_2)x^2+(-x_1-x_2-x_3)x+x_2x_3-x_1$ and
$$J=
\left[\begin{array}{rrr}
0&-1&0\\
-1&-1&-1\\
-1&x_3&x_2
\end{array}\right]
\qquad {\mbox{\rm{with \quad}}} 
J'=
\left[\begin{array}{rrr}
0&-1&0\\
-1&-1&-1\\
-1&-\frac{1}{2}&1
\end{array}\right]
$$
On the other hand
$$\adj{xI-N}=
\left[\begin{array}{ccc}
x^2-1&x+1&x-1\\
-\frac{1}{2}x-\frac{1}{2}&x^2+x+\frac{1}{2}&-\frac{1}{2}\\
-\frac{1}{2}x+\frac{1}{2}&-\frac{1}{2}&x^2-x+\frac{1}{2}
\end{array}\right]$$
Hence $p_{i_1,j_1}(x)=0x^2+x+1$, 
$p_{i_2,j_2}(x)=x^2+x+\frac{1}{2}$ 
and 
$p_{i_3,j_3}(x)=0x^2+x-1$. The coefficients of these polynomials
appear as the columns of $(-J')$.
}
\end{example}

The following lemma, formalizing another observation made in~\cite{P}, 
is a corollary of Lemma~\ref{P1}.

\begin{lemma}\label{P2}
The Jacobian matrix $J'$ described in line~$(\ref{Jacob})$
is nonsingular if and only if the set
of polynomials $\{p_{i_1,j_1}(x),p_{i_2,j_2}(x),\ldots,p_{i_n,j_n}(x)\}$ is linearly independent.
\end{lemma}

The next result was inspired by a careful reading
of a proof of Pereira,~\cite[Theorem 2.2]{P}, and in particular,
trying to determine a converse of Pereira's result.  

\begin{theorem}\label{big} Let $\A$ be a $n \times n$ pattern.
Suppose that one can
show that $\A$ is SAP using the Nilpotent-Jacobian method.
If $N$ is the nilpotent matrix in $Q(\A)$ used in the Nilpotent-Jacobian
method, then $N$ has index $n$.
\end{theorem}

\begin{proof} 
By our hypotheses, we can prove that the pattern $\A$ is SAP by
using the Nilpotent-Jacobian method.  Thus, by Theorem \ref{NJ},
we can find a nilpotent matrix $N \in Q(\A)$ that has the required
properties.  We wish to show that the index of $N$ is $n$.

As in~\cite[Theorem 2.2]{P}, we consider the vector space 
\[V = \operatorname{span}\{\adj{xI-N} ~|~ x \neq 0\}.\]
If $k$ is the index of $N$, then 
arguing as in~\cite{P}, we have $\dim V = k $ since 
\[\operatorname{span}\{\adj{xI-N} ~|~ x \neq 0\}
= \operatorname{span}\{(xI-N)^{-1} ~|~ x \neq 0\} 
= \operatorname{span}\{N^i\}_{i=0}^{k-1}
.\]
Our strategy is to show that $k = n$ by first finding
an $n$-dimensional vector space $S$, and then showing that 
$V = S$.

By Lemma~\ref{P2}, the polynomials $\{p_{i_1,j_1}(x),\ldots,p_{i_n,j_n}(x)\}$ form
a basis for $P_{n-1}$, the set of polynomials of degree at most $n-1$.

Let $W_x = \adj{xI-N}$.  Since each entry of $W_x$ is a polynomial
of degree at most $n-1$, each entry of $W_x$ can be rewritten in terms
of the basis  $\{p_{i_1,j_1}(x),\ldots,p_{i_n,j_n}(x)\}$.  In other words, we can
find matrices $D_1,\ldots,D_n$  with entries in $\mathbb{R}$ such that
\[W_x = p_{i_1,j_1}(x)D_1 + \cdots + p_{i_n,j_n}(x)D_n.\]

We claim that the matrices $\{D_1,\ldots,D_n\}$ are linearly independent.
In particular, 
the $(i_k,j_k)$-th entry of $W_x$ is $p_{i_k,j_k}(x)$, hence
the $(i_k,j_k)$-th entry of $D_k$ is $1$, but for all $(i_l,j_l)$ with
$l \neq j$, we have $(D_k)_{i_l,j_l} = 0$.  
As a consequence, the matrices must be linearly independent.

Each polynomial $p_{i_k,j_k}(x)$ can be written as 
\[p_{i_k,j_k}(x) = c_{k,0} + c_{k,1}x + \cdots + c_{k,n-1}x^{n-1}.\]
Thus, if $W_x  = p_{i_1,j_1}(x)D_1 + \cdots + p_{i_n,j_n}(x)D_n$, we 
can rewrite $W_x$ as 
\[W_x = x^{n-1}E_{n-1} + x^{n-2}E_{n-2} + \cdots + xE_1 + E_0\]
where 
\[E_l = c_{1,l}D_1 + c_{2,l}D_2 + \cdots + c_{n,l}D_n  ~~\mbox{
for $l=0,\ldots,n-1$.}\]
This is simply a mater of expanding $W_x$ and regrouping.

Let $S=\operatorname{span}\{E_0,\ldots,E_{n-1}\}$.

We now claim that the matrices $\{E_0,\ldots,E_{n-1}\}$ are linearly 
independent. Suppose 
that $a_0E_0 + \cdots +a_{n-1}E_{n-1} = 0$ (where $0$ denotes
the zero matrix of size $n$).
Note that this would imply that $b_1D_1 + \cdots + b_nD_n = 0$
for 
\[
\begin{bmatrix}b_1 & b_2 & \cdots & b_n
\end{bmatrix}
=
-\begin{bmatrix}a_{n-1} & a_{n-2} & \cdots & a_0\end{bmatrix}
J'
\]
where, by Lemma~\ref{P1}, $J'$ is the nonsingular Jacobian matrix. 
So $a_0E_0 + \cdots +a_{n-1}E_{n-1} = 0$ if and only if $a_0 = \cdots = a_{n-1}
= 0$.
Hence the matrices are linearly independent and $\dim S=n$.

For each nonzero $x$, we have shown that $W_x = \adj{xI - N}$
can be written has $W_x = x^{n-1}E_{n-1} + \cdots + xE_1 + E_0 \in S$.
Hence $V =\operatorname{span}\{\adj{xI-N} ~|~ x \neq 0\} \subseteq  S$. 

We now show that $S \subseteq V$.   Take $M \in S$.  So, there exists 
constants $r_1,\ldots,r_n$ such that 
\begin{equation}\label{eqn1}
M = r_0E_0 + \cdots + r_{n-1}E_{n-1}.
\end{equation}
We want to show that $M \in V$, i.e., it is enough to show
that we can find constants $d_1,d_2,\ldots,d_n$ 
and nonzero constants $z_1,z_2,\ldots,z_n$ 
such that
\begin{equation}\label{eqn2}
M = d_1W_{z_1} + d_2W_{z_2} + \cdots + d_nW_{z_n} ~~ \mbox{with
$W_{z_i} = \adj{z_iI-N}$} 
\end{equation}
As noted above, each matrix $W_{z_i}$ can be written as
\[W_{z_i} = z_i^{n-1}E_{n-1} + z_i^{n-2}E_{n-2} + \cdots + z_iE_1 + E_0.\]
So, if we expand out ($\ref{eqn2}$) and then compare to
($\ref{eqn1}$), we need to be able to solve 
for constants $d_1,d_2,\ldots,d_n$ 
and nonzero constants $z_1,z_2,\ldots,z_n$
in the system 
\begin{eqnarray*}
d_1z_1^{n-1} + d_2z_2^{n-1} + \cdots + d_nz_n^{n-1} & = & r_{n-1} \\
d_1z_1^{n-2} + d_2z_2^{n-2} + \cdots + d_nz_n^{n-2} & = & r_{n-2} \\
& \vdots & \\
d_1z_1 + d_2z_2 + \cdots + d_nz_n & = & r_1 \\
d_1 + d_2 + \cdots + d_n & = & r_0. \\
\end{eqnarray*}
But this system of equations is a Vandermonde system of equations,
i.e.,
\[
\begin{bmatrix}
z_1^{n-1} & z_2^{n-1} & \cdots & z_n^{n-1} \\
z_1^{n-2} & z_2^{n-1} & \cdots & z_n^{n-1} \\
\vdots & \vdots & \vdots & \vdots \\
z_1 & z_2 & \cdots & z_n \\
1 & 1 & \cdots & 1 \\
\end{bmatrix}
\begin{bmatrix}
d_1 \\
d_2 \\
\vdots\\
d_{n-1} \\
d_n
\end{bmatrix}
= 
\begin{bmatrix}
r_{n-1} \\
r_{n-2} \\
\vdots \\
r_1 \\
r_0
\end{bmatrix}.\]
Because the determinant of the Vandermonde matrix is given by
$\prod_{1 \leq i < j \leq n} (z_j -z_i)$, 
the 
system can be solved for
$d_1,d_2,\ldots,d_n$ 
given any distinct choices of $z_1,\ldots,z_n$.
Therefore $M\in V$ and $S\subseteq V$.

We have thus shown that the vector spaces $S$ and $V$ are the
same, and 
$k=\dim V=\dim S =n$. Therefore
the index of the nilpotent matrix $N$ is $n$.
\end{proof}

Part of our interest in this result is the following corollary:

\begin{corollary}
Let $\A$ be an $n \times n$ sign pattern.  If $Q(\A)$ 
has no nilpotent matrix of index $n$, then the Nilpotent-Jacobian
method cannot be used to prove that $\A$ is SAP.
\end{corollary}

We have shown that it is a necessary condition that an $n\times n$
pattern $\A$ allow a nilpotent matrix of index $n$ in order for
the Nilpotent-Jacobian method to successfully 
demonstrate that $\A$ is spectrally arbitrary. 
Pereira~\cite[Theorem 2.2]{P} showed that if a pattern $\A$ has
at most $n-2$ zero entries, all of which are on the diagonal, 
then it is sufficient to check if $\N$ allows full index to
demonstrate a pattern is spectrally arbitrary. As noted
already by Pereira~\cite[Example 1.1]{P}, allowing full index is not sufficient
for a pattern to be spectrally arbitrary. 

%
%
%


\begin{corollary}
The pattern $\A_2(\Z_s)$ allows full index for all $s\geq 3$.
\end{corollary}

\begin{proof} Let $s\geq 3$.
We note that 
the Nilpotent-Jacobian method was employed in~\cite[Theorem 4.4]{MTvdD} 
to show that superpatterns of certain star patterns (including $\Z_s$) were spectrally arbitrary.
Hence by Theorem~\ref{big}, $\Z_s$ allows full index. The result now follows from
Corollary~\ref{starfull}.
\end{proof}

\section{Concluding Comments and Open Problems}

We have shown that the pattern $\A_2(\Z_s)$ is spectrally arbitrary for all $s=3,4,$ and $5.$
We expect that the
pattern $\A_2(\Z_s)$ is spectrally arbitrary for all $s>5$, but we leave this as an open question.
Given $s\geq 3$, we observed that $A_m(\Z_s)$ is potentially
nilpotent for $m\geq 2$ and that $A_m(\Z_s)$ allows full index for $m=2$. It would be interesting to determine if
such patterns can allow full index when $m>2$ and if any such pattern is spectrally arbitrary.

In this paper we have demonstrated that
 some classes of tree patterns (as well as some other constructions)
 are potentially nilpotent. It would be interesting 
to eventually classify potentially nilpotent tree patterns, or even
potentially nilpotent balanced tree patterns. This may be a difficult problem, since even determining
potential nilpotence for the tridiagonal
pattern $\T_n$ is still open for $n>16$. 

We have noted that a couple of the balanced tree patterns
are spectrally arbitray. Much work could be done 
to develop more tools to determine what makes a balanced tree pattern
spectrally arbitrary. Such tools might shed some light
on the tridiagonal
pattern $\T_n$, to determine if $\T_n$ is spectrally
aribitrary for $n>16$.

It is an open question  (see, for example~\cite{Y}) whether every irreducible pattern that is
spectrally arbitrary can be shown to be spectrally arbitrary
using the Nilpotent-Jacobian method. Theorem~\ref{big} suggests
that one way to answer the question in the negative is to find a spectrally
arbitrary pattern which does not allow a nilpotent matrix
of full index. On the other hand, one might ask if every
spectrally arbitrary pattern allows a nilpotent matrix of full index.

\noindent
{\bf Acknowledgments.} 
The research was supported in part by
an NSERC USRA and NSERC Discovery grants.
The third author
thanks Redeemer University College for its hospitality while working
on this project.


\begin{thebibliography}{99}  

\bibitem{Britz}
T. Britz, J.J. McDonald, D.D. Olesky, and P. van den Driessche,
Minimal spectrally arbitrary sign patterns.
\emph{SIAM J. Matrix Anal. Appl.} {\bf 26} (2004), 257--271. 


\bibitem{Catral} M. Catral, D.D. Olesky, and P. van den Driessche,
Allow problems concerning sprectral properties of
sign pattern matrices: A survey.  \emph{Linear Algebra Appl.}
{\bf 430} (2009), 3080--3094.


\bibitem{CaversKimShader} M.S. Cavers, I.-J. Kim, B. Shader, and K.N. Vander Meulen,
On determining minimal spectrally arbitrary patterns.
\emph{Electron. J. Linear Algebra} {\bf 13} (2005), 240--248. 


\bibitem{CV05}
M.S. Cavers and K.N. Vander Meulen,
Spectrally and inertially arbitrary sign patterns.
\emph{Linear Algebra Appl.} {\bf 394} (2005), 53--72.

\bibitem{Drew} J.H. Drew, C.R. Johnson, D.D Olesky, P. van den Driessche, 
Spectrally arbitrary patterns.  \emph{Linear Algebra Appl.}  {\bf 308}  (2000),  121--137.

\bibitem{Elsner} L. Elsner, D.D. Olesky, and P. van den Driessche, 
Low rank perturbations and the spectrum of a tridiagonal pattern. 
\emph{Linear Algebra Appl.} {\bf 374} (2003), 219--230.

\bibitem{EK}
C. Erickson, I.-J. Kim,
On nilpotence indices of sign patterns.
\emph{Commun. Korean Math. Soc.} {\bf 25}:1 (2010),  11--18.

\bibitem{EJ} C.A. Eschenbach, C.R. Johnson,
 Several open problems in qualitative matrix theory involving eigenvalue distribution.
 \emph{Linear and Multilinear Algebra} {\bf 24} (1988), 79--80.

\bibitem{GS} Y. Gao, Y. Shao, 
Inertia sets of symmetric 2-generalized star sign patterns.
{\emph{Linear and Multilinear Algebra}} {\bf{54}}:1 (2006), 27--35.

\bibitem{LiLi} H. Li, J. Li, 
On potentially nilpotent double star sign patterns.
\emph{Czech. Math. J.} {\bf{59}}:2 (2009), 489--501.


\bibitem{MTvdD}  
G. MacGillivray, R.M. Tifenbach, and P. van~den~Driessche. 
 Spectrally arbitrary star sign patterns. 
 {\em Linear Algebra Appl.} {\bf 400} (2005), 99--119.


\bibitem{P} R. Pereira, 
Nilpotent matrices and spectrally arbitrary sign patterns.  
\emph{Electron. J. Linear Algebra}  {\bf 16}  (2007), 232--236.

\bibitem{Y} A. Yielding,
Spectrally arbitrary zero-nonzero patterns,
PhD Thesis, Washington State University, 2009.



\end{thebibliography}
\end{document}